\documentclass{psp}

\usepackage{amsmath,amsfonts,amssymb}

\newtheorem{theorem}{Theorem}[section]

\newtheorem{proposition}[theorem]{Proposition}
\newtheorem{lemma}[theorem]{Lemma}

\newcommand{\be}{\beta}

\newcommand{\ga}{\gamma}

\newcommand{\BR}{\mathbb{R}}

\newcommand{\BQ}{\mathbb{Q}}
\newcommand{\BZ}{\mathbb{Z}}
\newcommand{\BN}{\mathbb{N}}
\newcommand{\A}{\mathcal{A}}

\newcommand{\D}{\mathcal{D}}

\newcommand{\x}{{\vec{x}}}
\newcommand{\y}{{\vec{y}}}
\newcommand{\z}{{\vec{z}}}

\renewcommand{\limsup}{\varlimsup}
\renewcommand{\liminf}{\varliminf}

\newcommand{\B}{\mathcal{B}}

\renewcommand{\theenumi}{(\roman{enumi})}

\title{Extremal sequences of polynomial complexity}

\author[Kevin G. Hare, Ian D. Morris and Nikita Sidorov]
{KEVIN G. HARE\\
Department of Pure Mathematics, University of Waterloo,\addressbreak Waterloo, Ontario, Canada N2L 3G1.
\nextauthor IAN D. MORRIS\\
Department of Mathematics, University of Surrey,\addressbreak Guildford GU2 7XH, United Kingdom.
\and\ NIKITA SIDOROV\\
School of Mathematics, University of Manchester, \addressbreak Oxford Road, Manchester M13 9PL, United Kingdom.}

\begin{document}

\maketitle

\begin{abstract}
The joint spectral radius of a bounded set of $d \times d$ real matrices is defined to be the maximum possible exponential growth rate of products of matrices drawn from that set. For a fixed set of matrices, a sequence of matrices drawn from that set is called \emph{extremal} if the associated sequence of partial products achieves this maximal rate of growth. An influential conjecture of J. Lagarias and Y. Wang asked whether every finite set of matrices admits an extremal sequence which is periodic. This is equivalent to the assertion that every finite set of matrices admits an extremal sequence with bounded subword complexity. Counterexamples were subsequently constructed which have the property that every extremal sequence has at least linear subword complexity. In this paper we extend this result to show that for each integer $p \geq 1$, there exists a pair of square matrices of dimension $2^p(2^{p+1}-1)$ for which every extremal sequence has subword complexity at least $2^{-p^2}n^p$.

Keywords: Sturmian word, joint spectral radius, subword complexity, extremal sequence. MSC codes: Primary 15A60; Secondary 37B10, 65K10, 68R15.
\end{abstract}

\section{Introduction}
Given a finite set of $d \times d$ real matrices, $\A = \{A_0,\ldots,A_{m-1}\}$, the {\em joint spectral radius} of $\A$ is defined as
\begin{equation}\label{specr}
\varrho(\A):=\lim_{n \to \infty} \max\left\{\left\|A_{i_1}\cdots A_{i_n}\right\|^{1/n} \colon i_j \in \{0,\ldots,m-1\}\right\}.
\end{equation}
This definition was first introduced by G.-C.~Rota and G.~Strang in 1960 \cite{RS}. The joint spectral radius may be understood as describing the maximum possible exponential growth rate of sequences of matrices drawn from $\A$, via the expression
\begin{equation}\label{soup}
\varrho(\A)= \sup\left\{\limsup_{n \to \infty} \|A_{x_n}\cdots A_{x_1}\|^{\frac{1}{n}} \colon (x_i)_{i=1}^\infty \in \{0,\ldots,m-1\}^{\BN}\right\}.\end{equation}
(Here $\limsup$ is the standard limit superior of a sequence.)
This supremum is always attained, and its value is independent of the choice of norm (see e.g. \cite{Jungers}). In this paper we investigate the problem of identifying the structure of those sequences $(x_i)$ which attain the supremum in \eqref{soup}, which we term \emph{extremal sequences}. We will be concerned with three distinct types of extremal sequence, which will be defined later in this section.

The investigation of extremal sequences of finite sets of matrices begins with an influential paper of J. Lagarias and Y. Wang \cite{LW}, in which it was asked whether every finite set of matrices admits an extremal sequence which is periodic. This question was also raised independently by L. Gurvits \cite{Gu} in a somewhat different form. The study of extremal sequences also arises in the study of control theory of  discrete linear inclusions.
In \cite{TeMa12}, it is studied under the name {\em most unstable switching law}, whereas in \cite{Ko} they are given the name {\em extremal  trajectories}.

In order to state our results, we will require some notation and terminology from combinatorics on words (see for example \cite{AS03}). Let $\Sigma_m = \{0, 1, \ldots, m-1\}$ be the finite alphabet with $m$ values.
We will omit $m$ if the size of the alphabet is not relevant.
We denote by $\Sigma^*$ the set of finite words $u = u_1 u_2 \cdots u_n$ over the alphabet $\Sigma$.
In this case we say that $u_1 u_2 \cdots u_n$ has length $n$.
We denote by $\Sigma^\omega$ the set of infinite words, or sequences, $x = x_1 x_2 \cdots$ over the alphabet $\Sigma$.

We define a metric on $\Sigma^\omega$ by setting $d(x_1 x_2 \cdots, y_1 y_2 \cdots) := 2^{-n}$ where $n$ is the unique integer such that $x_n \neq y_n$ and $x_i = y_i$ for all $i < n$. With respect to this topology, $\Sigma^\omega$ is compact and totally disconnected. We define the {\em shift transformation} $\sigma$ on $\Sigma^\omega$ by $\sigma(x_1 x_2 x_3 \cdots) := x_2 x_3 \cdots$, which is continuous.

We will say that $u_1 u_2 \cdots u_n$ is a {\em subword} (sometimes called {\em factor}) of an infinite word $x = x_1 x_2 \cdots$ if there exists a $k$ with $u_i = x_{k+i}$ for $i = 1, 2, \ldots, n$. We denote the {\em language of $x$} by $\mathcal L(x)$, the collection of all subwords of $x$. We denote the set of all subwords of $x$ which have length $n$ by $\mathcal L_n(x)$. The {\em subword complexity} of the sequence $x \in \Sigma^\omega$ is the function $P(x,\cdot) \colon \BN \to \BN$ defined by $P(x,n):=\#\mathcal L_n(x)$.

A sequence $x \in \Sigma^\omega$ will be called {\em recurrent} if every subword of $x$ occurs infinitely many times in $x$.
This is the case if and only if $x \in \overline{\{\sigma^kx \colon k \geq 1\}}$, where the closure is taken with respect to the metric defined above. Note that if $y \in \overline{\{\sigma^kx \colon k \geq 0\}}$, then $\mathcal L(y) \subseteq \mathcal L(x)$.

A sequence $x= x_1 x_2 \cdots \in \Sigma^\omega$ is called {\em periodic} if there exists a word
    $u_1 u_2 \cdots u_n$ such that $x_{mn+i}=u_i$ for every $m \geq 0$ and $i=1,\ldots,n$.
We say that $x$ is {\em eventually periodic} if $\sigma^kx$ is periodic for some integer $k \geq 0$.
The sequence $x$ is eventually periodic if and only if $P(x,n) \leq n$ for some $n \geq 1$, if and
    only if $P(x,n)=P(x,n+1)$ for some $n \geq 1$, if and only if $P(x,\cdot)$ is bounded (see e.g. \cite{Fogg}).
In particular, if $x$ is \emph{not} eventually periodic then necessarily $P(x,n) \geq n+1$ for every integer $n \geq 1$.

If $u=u_1 u_2 \cdots u_n \in \Sigma_2^*$ is a finite word over the alphabet $\Sigma_2 = \{0,1\}$, we write
    $|u|_1:=\#\{1 \leq i \leq n \colon u_i=1\}$.
We say that a sequence $x \in \Sigma_2^\omega$ is {\em balanced} if for every
    $n \geq 1$ we have $\big||u|_1-|v|_1\big| \leq 1$ for all $u,v \in \mathcal L_n(x)$.
If $x$ is balanced then there exists $\gamma \in [0,1]$ such that
    \[ \lim_{n \to \infty}\frac{\#\{1 \leq i \leq n \colon x_i=1\}}{n} = \gamma. \]
For each $\gamma \in [0,1]$ we define
    \[X_\gamma :=\left\{x \in \Sigma_2^\omega \colon x\text{ is recurrent and balanced, and }\lim_{n \to \infty}
     \frac{\#\{1 \leq j \leq n \colon x_j=1\}}{n} = \gamma\right\}.\]
An infinite sequence $x \in \Sigma_2^\omega$ satisfies $P(x,n)=n+1$ for all $n \geq 1$ if and only
    if $x \in \bigcup_{\gamma \in [0,1]\setminus\BQ} X_\gamma$, see for example \cite{Fogg,Loth}.
A balanced word that is not eventually periodic is also known as a {\em Sturmian word}. %Alternate constructions of Sturmian words are discussed later.

Having set out the necessary ideas from combinatorics on words, let us now return to matrices. Let $\A = \{A_0, A_1, \ldots, A_{m-1}\}$ be a set of $d \times d$ matrices.  Following \cite{hmst,M2011} we say that a sequence $x = x_1 x_2 \cdots \in \Sigma_m$ is {\em strongly extremal} for $\A$ if there exists $\delta>0$ such that
$\|A_{x_n}\cdots A_{x_1}\| \geq \delta\varrho(\A)^n$ for every $n \geq 1$.
We say that a sequence is {\em weakly extremal} for $\A$ if
\[
\lim_{n \to \infty} \|A_{x_n}\cdots A_{x_1}\|^{1/n} =\varrho(\A).
\]
We further say that a sequence $x_1 x_2 \cdots$ is {\em very weakly extremal} for $\A$ if
\[
\limsup_{n \to \infty} \|A_{x_n}\cdots A_{x_1}\|^{1/n} =\varrho(\A),
\]
a definition which is original with this paper.
We see that if $x$ is strongly extremal for $\A$, then it is weakly extremal for $\A$.
Similarly, if $x$ is weakly extremal for $\A$, then it is very weakly extremal for $\A$.
We refer to the above defined sequences collectively as {\em extremal sequences}. Note that whether or not a given sequence belongs to one of these classes is independent of the choice of norm used in the definition. However, throughout this paper $\|\cdot\|$ will always denote the Euclidean norm.

Following \cite{LW}, a set of matrices $\A$ is said to have the {\em finiteness property} if there exists an eventually periodic extremal sequence for $\A$. It was shown in \cite{BTV} that for the one-parameter family $\{\A_\alpha \colon \alpha \in [0,1]\}$, where each pair $\A_\alpha:=\{A_0^{(\alpha)},A_1^{(\alpha)}\}$ is given by
\begin{equation}\label{btvmat}
A_0^{(\alpha)}=\begin{pmatrix} 1 & 1 \\ 0 & 1 \end{pmatrix},\qquad A_1^{(\alpha)}=\alpha\begin{pmatrix} 1 & 0 \\ 1 & 1 \end{pmatrix},
\end{equation}
there exist uncountably many $\alpha$ such that the pair $\A_\alpha$ does not have the finiteness property. Subsequent investigation in \cite{hmst} established the following theorem:

\begin{theorem}\label{sturm}
There exists a continuous, non-decreasing surjection $\mathfrak{r} \colon [0,1] \to [0,\frac{1}{2}]$ such that for each $ \alpha \in [0,1]$, the set $\A_\alpha$ defined above has the following properties:
\begin{enumerate}
\item
If $x \in X_{\mathfrak{r}(\alpha)}$ then $x$ is recurrent and strongly extremal for $\A_\alpha$. Conversely if $x \in \Sigma_2^\omega$ is recurrent and strongly extremal for $\A_\alpha$, then $x \in X_{\mathfrak{r}(\alpha)}$.
\item
If $x \in \Sigma_2^\omega$ is weakly extremal for $\A_\alpha$, then
\[\lim_{n \to \infty} \frac{1}{n}\sum_{j=0}^{n-1} \inf\left\{d(z,\sigma^jx) \colon z \in X_{\mathfrak{r}(\alpha)}\right\}=0.\]
\end{enumerate}
\end{theorem}

This result was applied in \cite{hmst} to give an explicit formula for a real number $\alpha_* \in (0,1)$ such that $\A_{\alpha_*}$ does not have the finiteness property (for more details see \S\ref{sec:ex} below). Further research in \cite{MS} showed that $\mathfrak r^{-1}(\gamma)$ is a singleton for all irrational $\gamma$, and the set of all $\alpha \in [0,1]$ for which $\A_\alpha$ lacks the finiteness property has zero Hausdorff dimension.

Theorem~\ref{sturm} implies that there exists $\alpha_*$ such that the pair of matrices $\A_{\alpha_*}$ has the following properties: for every very weakly extremal $x \in \Sigma_2^\omega$ one has $P(x,n) \geq n+1$ for every $n \geq 1$; and there exists $z \in \Sigma_2^\omega$ which is strongly extremal for $\A$ and satisfies $P(z,n) \equiv n+1$. Note that positivity estimates in \cite{hmst} may be applied to construct extremal sequences for $\A_{\alpha_*}$ which have any unbounded admissible subword complexity: specifically, one may intersperse a suitable Sturmian sequence with non-Sturmian finite words in a sufficiently sparse manner that the partial products along the sequence still grow with exponential rate equal to the joint spectral radius, whilst ensuring that every finite word over $\{0,1\}$ exists infinitely many times somewhere in the sequence. As such there does not exist a general  upper bound on the complexity of weakly extremal sequences for $\A_{\alpha_*}$ other than the trivial bound $P(x,n) \leq 2^n$ which applies to all sequences $x \in \Sigma_2^\omega$. More generally, if $\A=\{A_0,\ldots,A_{m-1}\}$ is any set of invertible $d \times d$ matrices which does not have a common invariant subspace other than $\{0\}$ and $\BR^d$, then the existence of Barabanov norms (see e.g. \cite{Ba,Wirth}) may be applied to construct weakly extremal sequences with subword complexity $m^n$ in a similar manner.

Theorem~\ref{sturm} nonetheless makes it reasonable for us to ask whether there exists a function $b \colon \BN \to \BN$, with $b(n)=o(2^n)$, such that for every pair of matrices $\{B_0,B_1\}$ there exists an extremal sequence with subword complexity $O(b(n))$. This may be seen as a continuation of the original question of J.~Lagarias and Y.~Wang in \cite{LW}, which is equivalent to the statement that every finite set of matrices admits an extremal sequence whose subword complexity is bounded. In particular one might ask whether every pair of matrices admits an extremal sequence $x$ such that $P(x,n) \leq n+1$ for every $n \geq 1$.  The main result of this paper implies that if such a universal complexity bound $b$ actually exists, then it is not a polynomial. We prove the following theorem:

\begin{theorem}\label{major}
For every integer $p \geq 1$, there exists a pair $\B:=\{B_0,B_1\}$ of
real square matrices of dimension $2^p(2^{p+1}-1)$ such that the
following properties hold:
\begin{enumerate}
\item
If $x \in \Sigma_2^\omega$ is very weakly extremal for $\B$, then \[P(x,n)
\geq \left(\left\lfloor \frac{n}{2^p}\right\rfloor+1\right)^p>2^{-p^2}n^p
 \]
for every $n \geq 1$.
\item
There exists $z \in \Sigma_2^\omega$ which is strongly extremal for
$\B$ and satisfies $P(z,n) \leq 2^p\left(\lceil \frac{n}{2^p}\rceil +2\right)^p$.
\end{enumerate}
\end{theorem}

We prove Theorem~\ref{major} in two separate stages. Firstly, we show that for each $p \geq 1$ one may construct a set of $2^p$ matrices, each of dimension $2^p \times 2^p$, for which all extremal sequences have subword complexity bounded below by $(n+1)^p$. This result, which we formally state as Proposition~\ref{main2} below, is achieved by taking suitable tensor products of pairs of matrices whose extremal sequences are known to have at least linear subword complexity.

Secondly, in Proposition~\ref{squidge} below we prove that for any set $\A=\{A_0,\ldots,A_{m-1}\}$ of $d$-dimensional matrices, one may construct a pair $\B = \{B_0, B_1\}$ of square matrices of dimension $(2m-1)d$ such that the subword complexity of the extremal sequences of $\B$ is asymptotically related to that of the extremal sequences of $\A$ in a precise manner. This construction is based on the proof of a theorem of R. Jungers and V. Blondel which states that a finite set of matrices possesses a periodic extremal sequence if and only if an associated {\em pair} of matrices has the same property \cite{JungersBlondel07,JungersBlondel08}. Since Jungers and Blondel's work only considers periodic extremal sequences, whilst we consider general extremal sequences, our analysis uses the same definitions but is significantly more technical in the manner in which these definitions are employed. The reader may verify that Theorem~\ref{major} follows from the combination of Propositions~\ref{main2} and \ref{squidge} in a straightforward manner. The proof of Theorem~\ref{major} is constructive, and we therefore provide an example of the pair of matrices arising in the case $p=2$ in Section~\ref{sec:ex} below.

\section{Construction of examples with polynomial complexity}
In this section we prove the following result:
\begin{proposition}\label{main2}
For each $p \geq 1$ there exists a set $\D$ consisting of $2^p$ matrices, each of dimension $2^p \times 2^p$, with the following properties: if $x \in \Sigma_{2^p}^\omega$ is very weakly extremal for $\D$, then $P(x,n)\geq (n+1)^p$ for all $n \geq 1$; and there exists a sequence $z \in \Sigma_{2^p}^\omega$ which is strongly extremal for $\D$ and satisfies $P(z,n) \equiv (n+1)^p$.
\end{proposition}

In order to prove Proposition~\ref{main2} we require several ancillary results. The following result was proved in \cite{M2011}: it ultimately derives from a lemma of Y.~Peres \cite{Peres}.
\begin{lemma}\label{heaviness}
Let $T \colon Z \to Z$ be a continuous transformation of a compact topological space, and let $f_n \colon Z \to \BR\cup\{-\infty\}$ be a sequence of upper semi-continuous functions such that $f_{n+k}(x) \leq f_n(T^kx)+f_k(x)$ for all $x \in Z$ and $n,k \geq 1$. Then there exists $z \in Z$ such that
\[\inf_{n \geq 1}\frac{1}{n} f_n(z) = \limsup_{n \to \infty}\sup_{x \in Z}\frac{1}{n}f_n(x)\]
and $z$ is recurrent with respect to $T$.
\end{lemma}

Here, for $z$ to be recurrent in $Z$ with respect to $T$ means that for every open neighbourhood of $z$ in $Z$, there exist infinitely many $k$ such that $T^k(z)$ is in this neighbourhood.
This is equivalent to our original definition of recurrent when $z \in \Sigma^\omega$,
    that every subword occurs infinitely often, by using
    $T = \sigma$ and $Z=\overline{\{\sigma^nz \colon n \geq 0\}}$.

Using this, we deduce the following result, which will also be used in the following section:

\begin{lemma}\label{strong}
Let $B_0,\ldots,B_{m-1}$ be $d \times d$ real matrices, and suppose that $x \in \Sigma_m^\omega$ is very weakly extremal for $\B:=\{B_0,\ldots,B_{m-1}\}$. Then there exists $y \in \Sigma_m^\omega$ which is recurrent and strongly extremal for $\B$, and satisfies $\mathcal L(y)\subseteq\mathcal L(x)$.
\end{lemma}
\begin{proof}
Let $Z:=\overline{\{\sigma^nx \colon n \geq 0\}}\subseteq \Sigma_m^\omega$, let $T \colon Z \to Z$ be the restriction of $\sigma$ to $Z$, and define $f_n(z):=\log\|B_{z_n}\cdots B_{z_1}\|$ for every $z=z_1 z_2 \cdots \in Z$, where we use the convention $\log 0:=-\infty$. Clearly each $f_n$ takes values in $\BR\cup\{-\infty\}$, and is continuous since it is a function of only finitely many coordinates. Note also that the hypothesis $f_{n+k}(z) \leq f_n(T^kz) + f_k(z)$ is satisfied for all $z \in Z$ and $n,k \geq 1$ by the submultiplicativity of operator norms. Using the hypothesis on $x$ together with the definition of the joint spectral radius of $\B$ we obtain
\[\log\varrho(\B) = \limsup_{n \to \infty} \frac{1}{n}\log\|B_{x_n}\cdots B_{x_1}\|\leq \limsup_{n \to \infty} \sup_{z \in Z} \frac{1}{n}\log \|B_{z_n}\cdots B_{z_1}\|\leq \log\varrho(\B).\]
It follows by Lemma~\ref{heaviness} that there exists a recurrent sequence $y \in Z$ such that
\[\inf_{n \geq 1}\frac{1}{n}\log\|B_{y_n}\cdots B_{y_1}\| = \log\varrho(\B),\]
and this equation implies that $y$ is strongly extremal for $\B$. Lastly, since $y \in \overline{\{\sigma^nx \colon n \geq 0\}}$, one may easily see that every subword of $y$ is also a subword of $x$ as required.
\end{proof}

Define
    \[ \Sigma_{2,p} = \underbrace{\Sigma_2 \times \Sigma_2 \times \cdots \times \Sigma_2}_p \]
as the Cartesian product of $p$ copies of $\Sigma_2$.
Define $\Sigma_{2,p}^\omega$ as the set of infinite words on $\Sigma_{2,p}$.
That is, \[ \x = \x_1 \x_2 \cdots \in \Sigma_{2,p}^\omega \]
    where \[ \x_i = (x_i^{(1)}, x_i^{(2)}, \ldots, x_i^{(p)}) \in \Sigma_{2,p}.\]
and each $x_i^{(j)} \in \Sigma_2$.
If we wish to specify only the infinite word coming from the $j$th coordinate, we will use the convention
    that \[ x^{(j)} = x_1^{(j)} x_2^{(j)} x_3^{(j)} \cdots \in \Sigma_2^\omega \]
With an abuse of notation, we will write
\begin{eqnarray*}
\x & = & \x_1 \x_2 \cdots \in \Sigma_{2,p}^\omega \\
   & = & x^{(1)} \times x^{(2)} \times \cdots \times x^{(p)} \in \Sigma_2^\omega \times \cdots \times
                        \Sigma_2^\omega. \\
\end{eqnarray*}
\begin{lemma}\label{p-sturm}
Let $\gamma_1,\ldots,\gamma_p \in (0,1)$ such that the set $\{1,\gamma_1,\ldots,\gamma_p\}$ is linearly independent over $\BQ$. Then every $\x \in X_{\gamma_1} \times \cdots \times X_{\gamma_p} \subset \Sigma_{2,p}^\omega$ satisfies $P(\x,n)=(n+1)^p$ for every $n \geq 1$.
\end{lemma}
\begin{proof}
Let $\x = x^{(1)} \times x^{(2)} \times \cdots x^{(p)} \in X_{\gamma_1} \times X_{\gamma_2} \times \cdots \times
    X_{\gamma_p} \subset \Sigma_{2,p}^\omega$.
Each sequence $x^{(j)} \in X_{\gamma_j}$, has complexity function $P(x^{(j)}, n) \equiv n+1$.
Every subword of $\x \in \Sigma_{2,p}^\omega$ of length $n$ corresponds to a $p$-tuple of subwords drawn
    respectively from $x^{(1)}, x^{(2)}, \ldots, x^{(p)} \in \Sigma_{2}^\omega$.
It follows that the number of subwords of $\x$ which have length $n$ cannot be greater than
    $\prod_{j=1}^p P(x^{(j)},n) = (n+1)^p$.
To prove that $P(\x, n)$ is precisely this amount we must show that every possible $p$-tuple of
    subwords of $x^{(1)} \times x^{(2)} \times \cdots \times x^{(p)}$ arises at some position in the
    sequence of $\x = \x_1 \x_2 \cdots$.
That is, if for each $j = 1, 2, \ldots, p$, the word $u_1^{(j)} u_2^{(j)} \cdots u_n^{(j)}$ is a
    subword of $x^{(j)}$, then there exists a $t \geq 0$ such that
    $x_{i+t}^{(j)} = u_{i}^{(j)}$ for all $i = 1, 2, \ldots, n$ and $j = 1, 2, \ldots, p$.

Let us define $T_{\gamma_j} \colon \BR/\BZ \to \BR/\BZ$ by $T_{\gamma_j}y:=y+\gamma_j$ for all $y \in \BR/\BZ$ and $j=1,\ldots,p$, and define $T_{\gamma_1,\ldots,\gamma_p} \colon (\BR/\BZ)^p \to (\BR/\BZ)^p$ by $T_{\gamma_1,\ldots,\gamma_p}:=T_{\gamma_1}\times \cdots \times T_{\gamma_p}$.
Since the set $\{1,\gamma_1,\ldots,\gamma_p\}$ is $\BQ$-linearly independent, the map $T_{\gamma_1,\ldots,\gamma_p}$ is a minimal transformation of the $p$-torus: that is, for every $(z_1,\ldots,z_p) \in (\BR/\BZ)^p$ and nonempty open set $V \subseteq (\BR/\BZ)^p$, there exists $N \in \BN$ such that $T_{\gamma_1,\ldots,\gamma_p}^N(z_1,\ldots,z_p) \in V$.
For a proof of this statement see e.g. \cite{KH}.
Now, since each $x^{(j)} \in X_{\gamma_j}$ is Sturmian, for each $j=1,\ldots,p$ there exists $z_j \in \BR/\BZ$ such that either
\[
x^{(j)}_i = \lfloor (i+1)\gamma_j +z_j\rfloor - \lfloor i\gamma_j +z_j\rfloor =\chi_{[1-\gamma_j,1)}(T^i_{\gamma_j}z_j)
\]
for all $i \geq 1$, or
\[
x^{(j)}_i \equiv \lceil (i+1)\gamma_j +z_j\rceil - \lceil i\gamma_j +z_j\rceil =\chi_{(1-\gamma_j,1]}
(T^i_{\gamma_j}z_k)
\]
for all $i \geq 1$ -- see for example \cite{AS03}. Here $\chi$ is the characteristic function define in the standard way as $\chi_A(z) = 1$ if $z \in A$  and $0$ otherwise.

In either case $\BR/\BZ$ decomposes into two disjoint half-open intervals $J^0_j, J^1_j$ with the property that for every $i \geq 1$, $x^{(j)}_i=1$ if and only if $T_{\gamma_j}^iz_j \in J^1_j$, and $x^{(j)}_i=0$ if and only if $T_{\gamma_j}^iz_j \in J^0_j$.
It follows that for any word $u_1 u_2 \cdots u_n$ over the alphabet $\{0,1\}$, we have $x^{(j)}_{\ell+i}=u_i$ for each $i=1,\ldots,n$ if and only if $T^\ell_{\gamma_j}z_j \in T_{\gamma_j}^{-1} J_j^{u_1} \cap \cdots \cap T_{\gamma_j}^{-n}J_k^{u_n}$ for some integer $\ell \geq 0$.

We may now show that $P(\x,n)=(n+1)^p$ for every $n \geq 1$.
Fix $n \geq 1$ and suppose that for each $j=1,\ldots,p$, we have $u^{(j)}_1 \cdots u^{(j)}_n$ is a word belonging to the language of $x^{(j)}$.
For each $j=1,\ldots,p$ it follows that the half-open interval $ T_{\gamma_j}^{-1}J_j^{u_1^{(j)}} \cap \cdots \cap T_{\gamma_j}^{-n}J_j^{u_n^{(j)}}$ is nonempty, and hence contains a nonempty open subset $U_j$.
Since $T_{\gamma_1,\ldots,\gamma_p}$ is minimal, there exists $\ell \geq 0$ such that $T_{\gamma_1,\ldots,\gamma_p}^\ell(z_1,\ldots,z_p)$ belongs to the open set $U_1 \times \cdots \times U_p \subset (\BR/\BZ)^p$.
We thus have $T^\ell_{\gamma_j}z_j \in U_j$ for each $j$ and therefore $x^{(j)}_{\ell+i}=u_i^{(j)}$ for $j=1,\ldots,p$ and $i=1,\ldots,n$.
The product word $(u^{(1)}_1, \ldots, u^{(p)}_1)(u^{(1)}_2, \ldots, u^{(p)}_2)\cdots(u^{(1)}_n,\ldots,u^{(p)}_n)$
therefore belongs to the language of $\x$ as claimed.
We conclude that $\#\mathcal L_n(\x)=\prod_{j=1}^p \# \mathcal L_n(x^{(j)})$ for every $n \geq 1$ and therefore $P(\x,n) =(n+1)^p$ as required.
\end{proof}

%%%%%%%%
The following standard result may be found in e.g. \cite[p.24]{Jungers}:
\begin{lemma}\label{boundedness}
Let $\B$ be a bounded nonempty set of $d \times d$ real matrices, and suppose that the only subspaces of $\BR^d$ which are preserved by all elements of $\B$ are $\BR^d$ and $\{0\}$. Then there exists a constant $K>0$ such that
\[\sup\{\|B_{i_n} \cdots B_{i_1}\| \colon B_i \in \B\} \leq K\varrho(\B)^n\]
for every integer $n \geq 1$.
\end{lemma}

We may now prove Proposition~\ref{main2}. We will construct a set $\D$ of matrices indexed over $\Sigma_{2,p}$, which differs from the statement of the proposition only in its notation. Let us choose $\gamma_1,\ldots,\gamma_p \in (0,\frac{1}{2})$ such that $\{1,\gamma_1,\ldots,\gamma_p\}$ is linearly independent over $\BQ$, and choose $\alpha_1,\ldots,\alpha_p \in (0,1)$ such that $\mathfrak{r}(\alpha_j)=\gamma_j$ for every $j=1,\ldots,p$. (Once each $\gamma_i$ has been specified, the choice of $\alpha_i$ is in fact unique.) Define $B_i^{(j)}:=A_i^{(\alpha_j)}$ for $j=1,\ldots,p$ and $i=0,1$ where $A^{(\alpha)}_i$ is as defined in \eqref{btvmat}, and let
\[
\D := \left\{D_{(x^{(1)}, \ldots, x^{(p)})}
         = \bigotimes_{j=1}^p B_{x^{(j)}}^{(j)} \colon (x^{(1)}, \cdots, x^{(p)}) \in \Sigma_{2,p}\right\},
\]
where $\otimes$ denotes the Kronecker product. Recall from e.g. \cite{HJ2} that if $G_1,\ldots,G_p$ and $H_1,\ldots,H_p$ are $d \times d$ matrices, then
\[\left(\bigotimes_{j=1}^p G_j\right)\left(\bigotimes_{j=1}^p H_j\right) = \bigotimes_{j=1}^p \left(G_jH_j\right)\]
and
\[
\left(\bigotimes_{j=1}^p G_j\right)^T =\left(\bigotimes_{j=1}^p G_j^T\right),\qquad \rho\left(\bigotimes_{j=1}^p G_j\right) =\prod_{j=1}^p\rho(G_j),
\]
where $\rho$ stands for the spectral radius of a matrix. Using these relations together with the identity $\|G\|=\sqrt{\rho(G^TG)}$ one may easily derive the useful identity
\[
\left\|\bigotimes_{j=1}^p G_j\right\| =\prod_{j=1}^p\|G_j\|
\]
for arbitrary $d \times d$ matrices $G_j$. We may thus compute
\begin{align*}
\varrho(\D)&=\lim_{n \to \infty}\max\left\{ \left\|D_{\x_n} \cdots D_{\x_1}\right\|^{1/n} \colon \x \in \Sigma_{2,p}^\omega \right\}\\
&=\lim_{n \to \infty}\max\left\{\prod_{j=1}^p\left\|B^{(j)}_{x_n^{(j)}}\cdots B^{(j)}_{x^{(j)}_{1}}\right\|^{1/n}\colon \x \in \Sigma_{2,p}^\omega\right\}\\
&=\lim_{n \to \infty} \left(\prod_{j=1}^{p} \left(\max\left\{\left\|B^{(j)}_{x^{(j)}_n} \cdots B^{(j)}_{x^{(j)}_1}\right\|^{1/n} \colon x_{i}^{(j)} \in \{0,1\}\right\}\right)\right)\\
&=\prod_{j=1}^n \varrho\left(\B^{(j)}\right).
\end{align*}
Now let us suppose that $\x \in \Sigma_{2,p}^\omega$ is very weakly extremal for $\D$. By Lemma~\ref{strong} we may find a recurrent sequence $\y \in \Sigma_{2,p}^\omega$  which is strongly extremal for $\D$ and satisfies $\mathcal L(\y)\subseteq\mathcal L(\x)$. For nonzero $\alpha$ the matrices in \eqref{btvmat} do not admit a common invariant subspace other than $\{0\}$ and $\BR^2$, and it follows from Lemma~\ref{boundedness} that there exists a constant $K>0$ such that
\[\max_{1 \leq j \leq p} \sup_{n \geq 1} \frac{1}{\varrho\left(\B^{(j)}\right)^{n}} \left\|B_{y_n^{(j)}}^{(j)} \cdots B_{y_1^{(j)}}^{(j)}\right\| \leq K.\]
Since $\y$ is strongly extremal for $\D$, it follows that
\begin{align*}
0&<\inf_{n \geq 1} \varrho(\D)^{-n}\|D_{\y_n} \cdots D_{\y_1}\|\\
&=\inf_{n \geq 1} \prod_{j=1}^p\left(\varrho\left(\B^{(j)}\right)^{-n} \left\|B_{y_{n}^{(j)}}^{(j)} \cdots B_{y_{1}^{(j)}}^{(j)}\right\|\right)\\
&\leq K^{p-1}\cdot\inf_{n \geq 1} \varrho\left(\B^{(j)}\right)^{-n}\left\|B_{y_{n}^{(j)}}^{(j)} \cdots B_{y_{1}^{(j)}}^{(j)}\right\|,
\end{align*}
for each $j=1,\ldots,p$, and therefore $y^{(j)}$ is strongly extremal for $\B^{(j)}$ for every $j=1,\ldots,p$. Since $\y$ is recurrent, each $y^{(j)}$ is clearly also recurrent, and it follows from Theorem~\ref{sturm} that $y^{(j)} \in X_{\gamma_j}$ for every $j$. We conclude that $\y \in X_{\gamma_1} \times \cdots \times X_{\gamma_p}$ and therefore $P(\y,n)=(n+1)^p$ for all $n \geq 1$ by Lemma~\ref{p-sturm}, which implies that $P(\x,n) \geq (n+1)^p$ as claimed.

To complete the proof of the proposition we must show that there exists a sequence $\z \in \Sigma_{2,p}^\omega$ which is strongly extremal for $\D$ and satisfies $P(\z,n)\equiv (n+1)^p$. Let $\z \in X_{\gamma_1} \times \cdots \times X_{\gamma_p}$ and write $\z_i=(z^{(1)}_i,\ldots,z^{(p)}_i)$ for every $i \geq 1$. By Lemma~\ref{p-sturm} we have $P(\z,n) \equiv (n+1)^p$, and since each sequence $z_1^{(j)} z_2^{(j)} \cdots$ belongs to $X_{\gamma_j}$ it follows from Theorem~\ref{sturm} that
\[\inf_{n \geq 1} \varrho\left(\B^{(j)}\right)^{-n}\left\|B^{(j)}_{z^{(j)}_n} \cdots B^{(j)}_{z^{(j)}_1}\right\|>0\]
for $j=1,\ldots,p$. By taking the product over $j=1,\ldots,p$ of this expression we deduce that $\inf_{n \geq 1} \varrho(\D)^{-n}\|D_{\z_n} \cdots D_{\z_1}\|>0$ and therefore $\z$ is strongly extremal for $\D$ as required. The proof is complete.

\section{Reduction to pairs of matrices}
In this section we prove the following result, which may be viewed as an extension of Theorem~3 of \cite{JungersBlondel08}:

\begin{proposition}\label{squidge}
Let $\A:=\{A_0,\ldots,A_{m-1}\}$ be a set of $d \times d$ real
matrices with nonzero joint spectral radius. Then there exists a pair
$\B:=\{B_0,B_1\}$ of real square matrices of dimension $(2m-1)d$ with
the following properties:
\begin{enumerate}
\item
For every $x \in \Sigma_2^\omega$ which is very weakly extremal for $\B$,
we may find $z \in \Sigma_m^\omega$ which is strongly extremal for
$\A$ and satisfies $P(z,\lfloor n/m \rfloor) \leq P(x,n)$ for all $n
\geq 1$.
\item
For every $z \in \Sigma_m^\omega$ which is very weakly extremal for
$\mathcal{A}$, we may find $x \in \Sigma_2^\omega$ which is strongly
extremal for $\B$ and satisfies $P(x,n) \leq mP(z,\lceil n/m
\rceil+1)$ for all $n \geq 1$.
\end{enumerate}
\end{proposition}

For the remainder of this subsection we fix matrices $A_0,\ldots,A_{m-1}$ such that the hypotheses of Proposition~\ref{squidge} are satisfied, and work towards the proof of the proposition. To construct the pair $B_0,B_1$ we use the following definition, which was first given in \cite{JungersBlondel08}.

 Let $V$ be the direct sum of $2m-1$ pairwise orthogonal copies of $\BR^d$, which we write as $V=\bigoplus_{i=0}^{2m-2}\BR^d$. For $i=0,\ldots,2m-2$ we shall use the notation $V_i \subseteq V$ to refer to the $d$-dimensional subspace of $\bigoplus_{i=0}^{2m-2}\BR^d$ consisting of those vectors which are zero in every co-ordinate of the direct sum except possibly the $i^{\mathrm{th}}$. We define two linear endomorphisms of $V$ by
\[B_0\left(v_0\oplus \cdots \oplus v_{2m-2}\right):=v_1 \oplus \cdots \oplus v_{2m-2} \oplus 0,\]
\[B_1\left(v_0 \oplus \cdots \oplus v_{2m-2}\right):= 0 \oplus \cdots \oplus 0 \oplus A_0v_0 \oplus A_1v_1 \oplus \cdots \oplus A_{m-1}v_{m-1}.\]
Thus $B_0$ is a surjection from $V$ onto $\bigoplus_{i=0}^{2m-3}V_i$ with kernel $V_0$, whilst $B_1$ maps $V$ into $\bigoplus_{i=m-1}^{2m-2}V_i$ with the subspace $\bigoplus_{i=0}^{m-2}V_i$ being included in its kernel.

The following lemma may be found in \cite{JungersBlondel08}; we provide a proof for the reader's convenience.
\begin{lemma}\label{squadge}
The joint spectral radius of $\B$ satisfies $\varrho(\B) \geq \varrho(\A)^{1/m}$.
\end{lemma}
\begin{proof}
Let $n \geq 1$ and $x_1,\ldots,x_n \in \{0,\ldots,m-1\}$ be arbitrary, and let $u:=0 \oplus \cdots \oplus 0 \oplus v \oplus 0 \oplus \cdots \oplus 0 \in V_{m-1}$. A simple calculation using the definition of $B_0,B_1$ shows that
\[
\bigl(B_0^{x_j}B_1B_0^{m-1-x_j}\bigr) v= 0 \oplus \cdots \oplus 0 \oplus A_{x_j}v \oplus 0 \oplus \cdots \oplus 0 \in V_{m-1}
\]
for any $j$. Hence the vector
\[
\left(B_0^{x_n}B_1B_0^{m-1-x_n}\right)\cdots \left(B_0^{x_2}B_1B_0^{m-1-x_2}\right) \left(B_0^{x_1}B_1B_0^{m-1-x_1}\right)v
\]
is equal to the vector
\[0 \oplus \cdots \oplus 0 \oplus \bigl(A_{x_n}\cdots A_{x_1}v \bigr) \oplus 0 \oplus \cdots \oplus 0 \in V_{m-1}.\]
It follows immediately that
\begin{align*}\varrho(\B) &\geq \limsup_{n \to \infty}\max\left\{\|B_{i_{mn}}\cdots B_{i_1}\|^{1/mn}\colon i_j \in \{0,1\}\right\}\\
&\geq \limsup_{n \to \infty}\max\left\{\|A_{i_{n}}\cdots A_{i_1}\|^{1/mn}\colon i_j \in \{0,\ldots,m-1\}\right\}\\&=\varrho(\A)^{1/m}\end{align*}
as required.
\end{proof}

The remainder of the proof of Proposition~\ref{squidge} deviates entirely from \cite{JungersBlondel08}. We next prove:

\begin{lemma}\label{scooge}
Let $x \in \Sigma_2^\omega$ be very weakly extremal for $\B$. Then there exists an integer $k \geq 0$ such that $y:=\sigma^kx$ is very weakly extremal for $\B$, and $B_{y_n}\cdots B_{y_1}V_{m-1} \neq \{0\}$ for every $n \geq 1$.
\end{lemma}
\begin{proof}
By the hypothesis of Proposition~\ref{squidge} we have $\varrho(\A)>0$, and hence $\varrho(\B)>0$ by Lemma~\ref{squadge}. Let $x = x_1 x_2 \cdots \in \Sigma_2^\omega$ be very weakly extremal. A simple calculation using submultiplicativity and the definition of $\varrho(\B)$ shows that $\sigma^rx$ is also very weakly extremal for every $r \geq 0$.

If for every $j=0,\ldots,2m-2$ there exists $n>0$ such that $B_{x_n}\cdots B_{x_1}V_j=\{0\}$, then for all large enough $n$ we have $B_{x_n}\cdots B_{x_1}=0$, which contradicts the fact that $x$ is very weakly extremal and $\varrho(\B)$ is nonzero. Let us therefore choose $V_j$ such that $B_{x_n}\cdots B_{x_1}V_j \neq \{0\}$ for every $n \geq 1$.

If $j=m-1$ then of course we may take $k:=0$ to prove the lemma. If $m-1<j \leq 2m-2$ then a simple calculation shows that $B_1B_0^rV_j=\{0\}$ for all integers $r$ such that $0 \leq r <j-m+1$. It follows that $x_i=0$ when $1 \leq i \leq j-m+1$,  since otherwise $B_{x_{j-m+1}}\cdots B_{x_1}V_j=\{0\}$ and the choice of $j$ is contradicted. We deduce that $B_{x_{j-m+1}}\cdots B_{x_1}V_j =B_0^{j-m+1}V_j\subseteq V_{m-1}$. Taking $k:=j-m+1$ yields $B_{y_n}\cdots B_{y_1}V_{m-1} \neq \{0\}$ for all $n \geq 1$, since if it were the case that $B_{y_n}\cdots B_{y_1}V_{m-1}=\{0\}$ for some $n$ then we would have $B_{x_{j-m+n+1}}\cdots B_{x_1}V_j=\{0\}$, which contradicts the definition of $j$.

Lastly let us suppose that $0 \leq j <m-1$. Since $B_0^{j+1}V_j=\{0\}$ it follows that there exists an integer $r$ such that $0 \leq r \leq j$ and $B_{x_{r+1}}\cdots B_{x_1}=B_1B_0^r$. We thus have $B_{x_{r+1}}\cdots B_{x_1}V_j \subseteq V_{j-r+m-1}$. Let $z:=\sigma^{r+1}x$. The sequence $z=z_1 z_2 \cdots$ then satisfies $B_{z_n}\cdots B_{z_1}V_{j-r+m-1} \neq \{0\}$  for all $n \geq 1$, since otherwise we would have $B_{x_{n+r+1}} \cdots B_{x_1}V_j =\{0\}$ which contradicts the choice of $j$. Since furthermore $m-1 \leq j-r+m-1 < 2m-2$, the sequence $z_1 z_2 \cdots$ falls within the scope of the arguments used in the previous paragraph, and it follows that $y:=\sigma^{j-r}z=\sigma^{j+1}x$ has the properties stipulated by the lemma. The proof is complete.
\end{proof}

\begin{lemma}\label{snurk}
Let $x=x_1 x_2 \cdots  \in \Sigma_2^\omega$ and suppose that for every $n \geq 1$ we have $B_{x_n}\cdots B_{x_1} V_{m-1} \neq \{0\}$. Then for every integer $n \geq 0$ precisely one of the symbols $x_{mn+1},\ldots,x_{m(n+1)}$ is equal to one, and the remainder are zero.
\end{lemma}
\begin{proof}
To see that exactly one of the symbols $x_1,\ldots,x_m$ equals one we argue as follows. If all of these symbols equal zero, then we have $B_{x_m}\cdots B_{x_1}V_{m-1}=B_0^mV_{m-1}=\{0\}$, contradicting the hypothesis. On the other hand, if more than one of these symbols equals one then $B_{x_m}\cdots B_{x_1} = MB_1B_0^rB_1B_0^s$ for some integers $r,s \geq 0$ with $r+s \leq m-2$ and some matrix $M$ (which may be the identity). One may easily verify from the definitions of $B_0$ and $B_1$ that in this case $B_1B_0^rB_1B_0^sV_{m-1}\subseteq B_1B_0^rV_{2m-2-s}=\{0\}$, which again contradicts the hypothesis of the lemma. We conclude that exactly one of the symbols $x_1,\ldots,x_m$ is equal to one as claimed.

A simple calculation shows that necessarily $B_{x_m}\cdots B_{x_1}V_{m-1} \subseteq V_{m-1}$, and it follows from this that the subspace $B_{x_{m+n}}\cdots B_{x_{m+1}}V_{m-1}$ is also not equal to $\{0\}$ for any $n \geq 1$. The sequence $\sigma^mx$ therefore also satisfies all of the initial hypotheses of the lemma, and by repeating the above arguments inductively we obtain the conclusion of the lemma.
\end{proof}

We may now prove Proposition~\ref{squidge}~(i). Let us suppose that $x \in \Sigma_2^\omega$ is very weakly extremal for $\B$. By Lemma~\ref{scooge}, we may find a sequence $y=\sigma^kx$ such that $B_{y_n}\cdots B_{y_1}V_m \neq \{0\}$ for every $n \geq 1$, and we clearly have $\mathcal L(y)\subseteq\mathcal L(x)$. By Lemma~\ref{snurk}, for every $n \geq 0$ exactly one of the symbols $y_{mn+1},\ldots,y_{m(n+1)}$ is equal to one. Let us define a new sequence $w\in \Sigma_m^\omega$ by setting each $w_i$ to be the unique integer $k \in \{0,\ldots,m-1\}$ such that $y_{mi-k}=1$. We thus have
\[B_{y_{mn}} \cdots B_{y_1} = \left(B_0^{w_n}B_1B_0^{m-1-w_n}\right) \cdots \left(B_0^{w_1}B_1B_0^{m-1-w_1}\right)\]
for every $n \geq 1$. The rule defining $w$ determines a bijective function from the set of all subwords of $y$ of length $mn$ beginning at positions congruent to $1$ modulo $m$ to the set of all subwords of $w$ which have length $n$, and it follows that $\#\mathcal L_{mn}(y) \geq \#\mathcal L_n(w)$ for every integer $n \geq 1$. Since the function $P(y, \cdot) \colon \BN \to \BN$ is increasing, we deduce that $P(y,n) \geq P(y,m\lfloor n/m \rfloor) \geq P(w,\lfloor n/m\rfloor)$ for every $n \geq 1$.

For convenience let us define $D_j:=B_0^jB_1B_0^{m-1-j}$ for $j=0,\ldots,m-1$, so that for each $n \geq 1$ we have
\[B_{y_{mn}} \cdots B_{y_1} = D_{w_n} \cdots D_{w_1}.\]
Define $A_j$ to be the zero matrix for all $j \in \BZ \setminus \{0,\ldots,m-1\}$. Subject to this convention we have
\[D_j\left(\bigoplus_{i=0}^{2m-2} v_i\right)=\bigoplus_{i=0}^{2m-2}A_{j+i-m+1}v_i\]
for every $j=0,\ldots,m-1$, and hence for each $n \geq 1$,

\begin{equation}\label{feqt}D_{w_n}\cdots D_{w_1}\left(\bigoplus_{i=0}^{2m-2} v_i\right)=\bigoplus_{i=0}^{2m-2}A_{w_n+i-m+1}\cdots A_{w_1+i-m+1}v_i.
\end{equation}
Now, since $u$ is very weakly extremal for $\B$, we have
\[ \limsup_{n \to\infty}\|D_{w_n}\cdots D_{w_1}\|^{1/n}=\limsup_{n \to \infty}\|B_{y_{nm}}\cdots B_{y_1}\|^{1/n}=\varrho(\B)^m,\]
and in view of \eqref{feqt} it follows that there exists at least one integer $k \in \BZ$ such that

\begin{equation}\label{foots}
\limsup_{n \to \infty} \|A_{w_n+k} \cdots A_{w_1+k}\|^{1/n} =\varrho(\B)^m.
\end{equation}
Since $\varrho(\B)>0$ this implies that $0 \leq w_i +k \leq m-1$ for every $i \geq 1$, so the sequence $\varpi$ defined by $\varpi_i:=w_i+k$ for all $i \geq 1$ belongs to $\Sigma_m^\omega$. Obviously we also have $P(\varpi,n)=P(w,n)$ for every integer $n \geq 1$.
Using the definition of the joint spectral radius of $\A$ together with \eqref{foots} we have
\[\varrho(\B)^m=\limsup_{n \to \infty} \|A_{w_n+k} \cdots A_{w_1+k}\|^{1/n} =
\limsup_{n \to \infty} \|A_{\varpi_n} \cdots A_{\varpi_1}\|^{1/n} \leq \varrho(\A),\]
 whereas Lemma~\ref{squadge} states that $\varrho(\B)^m \geq \varrho(\A)$. We conclude that $\varrho(\A)=\varrho(\B)^m$ and therefore $\varpi$ is very weakly extremal for $\A$. By Lemma~\ref{strong} we may find a new sequence $z \in \Sigma_m^\omega$ which is strongly extremal for $\A$ and satisfies $\mathcal L(z) \subseteq \mathcal L(\varpi)$, and it follows that
\[P(x,n) \geq P(y,n) \geq P(w,\lfloor n/m\rfloor) =P(\varpi,\lfloor n/m\rfloor) \geq P(z,\lfloor n/m \rfloor)\]
for every $n \geq 1$ as required by the statement of the proposition.

Let us now prove part~(ii) of the proposition. If $z \in
\Sigma_m^\omega$ is very weakly extremal for $\mathcal{A}$, then by Lemma~\ref{strong} we may choose $y \in \Sigma_m^\omega$ which is strongly
extremal for $\mathcal{A}$ and satisfies $\mathcal{L}(y) \subseteq
\mathcal{L}(z)$. For each $i \geq 1$ define $x_{mi-k}:=1$ if $y_i=k$,
and $x_j:=0$ for all other $j$. For each $n \geq 1$ the resulting
sequence $x \in \Sigma_2^\omega$ satisfies
\[B_{x_{mn}} \cdots B_{x_1} = \left(B_0^{y_n}B_1B_0^{m-1-y_n}\right)
\cdots \left(B_0^{y_1}B_1B_0^{m-1-y_1}\right)=D_{y_n}\cdots D_{y_1}.\]
We claim that $x$ is strongly extremal for $\mathcal{B}$. Choose
$\delta>0$ such that $\|A_{y_n}\cdots A_{y_1}\| \geq
\delta\varrho(\mathcal{A})^n$ for every $n \geq 1$. Clearly for each
$n \geq 1$ we have
\[\|B_{x_{mn}}\cdots B_{x_1}\|=\|D_{y_n}\cdots D_{y_1}\| \geq
\|A_{y_n}\cdots A_{y_1}\| \geq
\delta\varrho(\mathcal{A})^n=\delta\varrho(\mathcal{B})^{mn}.\]
Now fix a real number $K>0$ such that for all $k=1,\ldots,m-1$,
\[\max\left\{\|B_{i_k} \cdots B_{i_1}\| \colon i_j \in \{0,1\}\right\}
\leq K\varrho(\mathcal{B})^{m-k}.\]
If $n \geq 1$ satisfies $n=qm+r$ with $q \geq 0$ and $0 \leq r <m$, then since

\begin{align*}\|B_{x_{m(q+1)}} \cdots B_{x_1}\| &\leq \|B_{x_{m(q+1)}}
\cdots B_{x_{mq+r+1}} \|\cdot \|B_{x_{mq+r}} \cdots B_{x_1}\|\\
&\leq K\varrho(\mathcal{B})^{m-r}\|B_{x_{mq+r}} \cdots B_{x_1}\|\\
&=K\varrho(\mathcal{B})^{m-r}\|B_{x_n} \cdots B_{x_1}\|,
\end{align*}
we have
\[\|B_{x_n} \cdots B_{x_1}\| \geq
K^{-1}\varrho(\mathcal{B})^{r-m}\|B_{x_{m(q+1)}}\cdots B_{x_1}\| \geq K^{-1}\delta \varrho(\mathcal{B})^{qm+r}=
\delta K^{-1}\varrho(\mathcal{B})^n\]
and we conclude that $x$ is strongly extremal for $\mathcal{B}$ as claimed.

Let us now show that $x$ satisfies the required subword complexity
bound. For each $r=0,\ldots,m-1$ let us say that a subword $u_1\cdots
u_n$ of $x$ occurs in position $r$ modulo $m$ if there exists $q \geq
0$ such that $u_i=x_{qm+r+i}$ for $i=1,\ldots,n$. (Note that
$u_1\cdots u_n$ may occur in position $r$ modulo $m$ for several distinct $r$.) Similarly to the proof of part (i), the definition $x_{mi-k}:=1$ if and only if $y_i=k$, and $x_k:=0$ otherwise,
implies a bijection for each $n \geq 1$ between the set of all
subwords of $y$ of length $n$ and the set of subwords of $x$ of length~$mn$ which occur in position $0$ modulo $m$. In particular, the number of subwords of $x$ of the latter type is equal to precisely $P(y,n)$.

We claim that $P(x,mn+1) \leq mP(y,n+1)$ for all $n \geq 1$, for which
we use a counting argument. Fix $n \geq 1$, and suppose that for $1\leq j \leq mP(y,n+1)+1$ the word $u_1^{(j)} \cdots u_{mn+1}^{(j)}$ is a subword of $x$. We will show that at least two of the words $u^{(j)}_1 \cdots u^{(j)}_{mn+1}$ must coincide, which implies the required bound. Now, for each $j$ there exists $0 \leq r_j <m$ such that $u_1^{(j)} \cdots u_{mn+1}^{(j)}$ occurs in position $r_j$ modulo
$m$, and it follows by the pigeonhole principle that there exists $0 \leq r <m$ such that at least $P(y,n+1)+1$ of these words occur in position $r$ modulo $m$. It follows from this that there exist
$P(y,n+1)+1$ words $w^{(j)}_1 \cdots w^{(j)}_{m(n+1)}$ which occur as subwords of $x$ in position $0$ modulo $m$ and satisfy $w_{r+i}^{(j)}=u_i^{(j)}$ for all $i=1,\ldots,mn+1$. Since there are only $P(y,n+1)$ subwords of $x$ of length $m(n+1)$ which occur in position $0$ modulo $m$, it follows that two of the words $w^{(j)}_1 \cdots w_{m(n+1)}^{(j)}$ must be equal and hence the corresponding two words $u^{(j)}_1 \cdots u_{mn+1}^{(j)}$ are also equal. We conclude
that $\mathcal{L}_{mn+1}(x)$ contains at most $mP(y,n+1)$ distinct words, which proves the claim. It follows from the truth of the claim that for every integer $n \geq 1$, we have
\[P(x,n) \leq P(x,m\lceil n/m \rceil +1) \leq mP(y,\lceil n/m \rceil+1)
\leq mP(z,\lceil n/m \rceil+1)\]
as claimed. The proof is complete.

\section{Example}\label{sec:ex}

The function $\mathfrak{r}$ defined in Theorem~\ref{sturm} can be described explicitly using formulae derived in \cite{hmst,MS}, allowing us to make the construction in Proposition~\ref{main2} (and hence Theorem~\ref{major}) explicit. For example, let $\ga:=\frac{3-\sqrt5}2$ and $\be:=1-\frac{\sqrt2}2$. Clearly $\gamma,\beta \in (0,\frac{1}{2})$ and the set $\{1,\gamma,\beta\}$ is linearly independent over $\BQ$. By \cite[Theorem~1.1]{hmst},
\begin{align*}
\alpha_*:&=\mathfrak r^{-1}(\ga)=\lim_{n \to \infty}
\left(\frac{\tau_n^{F_{n+1}}}{\tau_{n+1}^{F_n}}\right)^{(-1)^n}= \prod_{n=0}^\infty \left(1-\frac{\tau_{n-2}}{\tau_{n-1} \tau_n}\right)^{(-1)^n F_{n+1}}\\
&=0.749326546330367557943961948091344672091327\ldots,
\end{align*}
where $(F_n)$ is the Fibonacci sequence defined by $F_0:=F_1:=1$ and $F_{n+1}:=F_n+F_{n-1}$ for all $n\ge1$, and $(\tau_n)$ is defined by $\tau_{-2}:=1$, $\tau_{-1}:=\tau_0:=2$ ,and $\tau_{n+1}:=\tau_n\tau_{n-1}-\tau_{n-2}$ for all $n\ge0$.
Similarly, by \cite[Theorem~2.5]{MS} and Remarks~8.8 and 8.10 from the same paper,
\begin{align*}
\alpha_{**}:&=\mathfrak r^{-1}(\be)=\lim_{n \to \infty}
\left(\frac{t_n^{G_{n+1}}}{t_{n+1}^{G_n}}\right)^{(-1)^n}= \prod_{n=0}^\infty
\left(\frac{t_n^2t_{n-1}}{t_{n+1}}
\right)^{(-1)^nG_n}\\
&=0.569279286584142330986485601616004654998409\ldots,
\end{align*}
where the sequences $(G_n)$, $(t_n)$ are defined firstly by $G_0:=1, G_1:=2$ and $G_{n+1}=2G_n+G_{n-1}$ for all $n \geq 1$, and secondly by $t_{-2}:=1$, $t_{-1}:=t_0:=2$ and $t_{n+1}:=t_n^2t_{n-1}-\frac{t_n^2}{t_{n-1}}- \frac{t_n t_{n-2}}{t_{n-1}}-t_{n-1}$ for all $n\ge0$.
If we now define
\[
D_0:=\begin{pmatrix} 1&1&1&1 \\ 0&1&0&1 \\ 0&0&1&1 \\ 0&0&0&1 \end{pmatrix},\qquad D_1:=\alpha_* \begin{pmatrix} 1&1&0&0 \\ 0&1&0&0 \\ 1&1&1&1 \\ 0&1&0&1 \end{pmatrix},\]
\[D_2:=\alpha_{**}\begin{pmatrix} 1&0&1&0 \\ 1&1&1&1 \\ 0&0&1&0 \\ 0&0&1&1 \end{pmatrix},\qquad D_3:= \alpha_*\alpha_{**} \begin{pmatrix} 1&0&0&0 \\ 1&1&0&0 \\ 1&0&1&0 \\ 1&1&1&1 \end{pmatrix},
\]
then the proof of Proposition~\ref{main2} shows that every very weakly extremal sequence of the set $\D:=\{D_0,\ldots,D_3\}$ has subword complexity bounded below by $(n+1)^2$. If we further define a pair of $28 \times 28$ matrices $\B:=\{B_0,B_1\}$ by
\[B_0:=\begin{pmatrix}
0&I&0&0&0&0&0 \\
0&0&I&0&0&0&0 \\
0&0&0&I&0&0&0 \\
0&0&0&0&I&0&0 \\
0&0&0&0&0&I&0 \\
0&0&0&0&0&0&I \\
0&0&0&0&0&0&0
 \end{pmatrix},
\qquad
B_1:=\begin{pmatrix}
0&0&0&0&0&0&0 \\
0&0&0&0&0&0&0 \\
0&0&0&0&0&0&0 \\
D_0&0&0&0&0&0&0 \\
0&D_1&0&0&0&0&0 \\
0&0&D_2&0&0&0&0 \\
0&0&0&D_3&0&0&0
\end{pmatrix}\]
where $I$ denotes the $4 \times 4$ identity matrix, then by Proposition~\ref{squidge}, every very weakly extremal sequence of $\B$ has subword complexity bounded below by $n^2/16$.

\section{Comments and Further Questions}

The result of Theorem~\ref{major} makes it natural for us to ask the following question:

\smallskip\noindent
{\bf Question 1.} Does there exist a set of matrices $\B=\{B_0,\ldots,B_{m-1}\}$ such that for every very weakly extremal sequence $x \in \Sigma_m$,
\begin{equation}\label{meh}\limsup_{n \to \infty} \frac{\log P(x,n)}{\log n} = +\infty?\end{equation}

\smallskip

By Theorem~\ref{major} we know that for each $p$ we can construct a
    pair of matrices $\B$ such that for every very weakly extremal sequence,
\[
\liminf_{n \to \infty} \frac{\log P(x,n)}{\log n} \geq p
\]
and that there exists a strongly extremal sequence such that
\[
\lim_{n \to \infty} \frac{\log P(x,n)}{\log n} = p.
\]
Hence, one way to interpret this question is: can we find a $\B$ such that the subword complexity of $\B$ must
    necessarily be greater than polynomial?

If the answer to Question~1 is negative, then it follows from arguments in \cite{Madv} that for each $\A=\{A_0,\ldots,A_{m-1}\}$ there exists a constant $\delta>0$ such that
\[\varrho(\A)=\max_{1 \leq k \leq n} \max\left\{\rho\left(A_{i_k} \cdots A_{i_1}\right)^{\frac{1}{k}} \colon 0 \leq  i_j <m \right\} + O\left(\exp(-n^{\delta})\right)\]
in the limit as $n \to \infty$. This result, if true, would substantially strengthen the main theorem in that paper.

Note that by Proposition~\ref{squidge} above, for Question~1 to have a positive answer it is both necessary and sufficient that there exist a \emph{pair} of matrices $\B$ all of whose very weakly extremal sequences satisfy \eqref{meh}. Furthermore, by Lemma~\ref{strong} above, the truth or otherwise of Question~1 remains unchanged if ``very weakly extremal'' is replaced with ``strongly extremal''. Analogous comments also apply to the following stronger question:

\smallskip\noindent
{\bf Question 2.}
Does there exist a set of matrices $\B=\{B_0,\ldots,B_{m-1}\}$ such that for every very weakly extremal sequence $x \in \Sigma_m^\omega$,
\begin{equation}\label{entropy}\lim_{n \to \infty} \frac{1}{n}\log P(x,n)>0?\end{equation}

\smallskip

If this were true, there would exist a $\B$ and a $c > 1$
    such that the subword complexity of $\B$ is bounded below by $c^n$.
Note that the limit in \eqref{entropy} will always exist by subadditivity. By \cite[Theorem~2.3]{M2011}, a strongly extremal sequence of a set of matrices $\A=\{A_0,\ldots,A_{m-1}\}$ cannot have dense orbit in $\Sigma_m^\omega$ unless \emph{every} sequence is strongly extremal for $\A$. It follows that there must always exist at least one very weakly extremal sequence $x \in \Sigma_m^\omega$ such that the limit in $\eqref{entropy}$ is strictly less than $\log m$.

Lastly, we remark that the dimension $2^p(2^{p+1}-1)$ which appears in Theorem \ref{major} seems very unlikely to be optimal. By either eliminating symmetries in the use of exterior products in the proof of Proposition \ref{main2}, or refining the construction of the block matrices in the proof of Proposition \ref{squidge}, it might be possible to reduce this dimension somewhat whilst retaining in outline the method used in this paper (but with a longer proof). More generally, we ask the following question:

\smallskip\noindent
{\bf Question 3.}
For each integer $p \geq 1$, how large is the smallest integer $d(p)$ for which there exists a pair of $d(p) \times d(p)$ matrices $\B=\{B_0,B_{1}\}$ such that for every very weakly extremal sequence $x \in \Sigma_m^\omega$,
\[\liminf_{n \to \infty} \frac{\log P(x,n)}{\log n}\geq p?\]

\smallskip

In view of Proposition \ref{squidge} a positive answer to Question 2 would imply that $d(p)$ is bounded with respect to $p$. In any event we do not anticipate that Question 3 will be easy to answer.
%\bibliographystyle{plain}
%\bibliography{paper}

\end{document}